\documentclass[11pt]{article}
\usepackage[margin=1in]{geometry}
\usepackage{amsthm,amsmath}
\usepackage{latexsym,amssymb,amsmath}
\usepackage{graphicx,float,color,fancybox,shapepar,setspace,hyperref}
\usepackage{subfigure}
\usepackage{pgf,tikz}
\usepackage[affil-it]{authblk}
\usepackage{indentfirst}
\usepackage{hyperref}
\usepackage[section]{placeins}
\usepackage{todonotes}
\hypersetup
{
colorlinks=true,
linkcolor=blue,
filecolor=blue,
urlcolor=blue,
citecolor=cyan,
}
\usetikzlibrary{arrows}
\newtheorem{thm}{Theorem}[section]
\newtheorem{conj}[thm]{Conjecture}

\newtheorem{lem}[thm]{Lemma}%[section]
\newtheorem{prop}[thm]{Proposition}%[section]
\newtheorem{defn}[thm]{Definition}%[section]

\def\~{\sim}

\newcommand{\FS}{\mathsf{FS}}

\begin{document}

\title{Connectedness of friends-and-strangers graphs of complete bipartite graphs and others}

\author{Lanchao WANG, Junying LU, Yaojun CHEN\thanks{Corresponding author. 
 Email: \href{mailto://yaojunc@nju.edu.cn}{ yaojunc@nju.edu.cn}.
}}
 \affil{ { \small {Department of Mathematics, Nanjing University, Nanjing 210093, China}}}
\date{}
\maketitle

\begin{abstract}

Let $X$ and $Y$ be any two graphs of order $n$. The friends-and-strangers graph $\mathsf{FS}(X,Y)$ of $X$ and $Y$ is a graph with vertex set consisting of all bijections $\sigma :V(X) \mapsto V(Y)$, in which two bijections $\sigma$, $\sigma'$ are adjacent if and only if they differ precisely on two adjacent vertices of $X$, and the corresponding mappings are adjacent in $Y$. The most fundamental question that one can ask about these friends-and-strangers graphs is whether or not they are connected.  Let $K_{k,n-k}$ be a complete bipartite graph of order $n$. In 1974, Wilson characterized the connectedness of $\mathsf{FS}(K_{1,n-1},Y)$ by using algebraic methods.
In this paper, by using combinatorial methods,  we investigate the connectedness of  $\mathsf{FS}(K_{k,n-k},Y)$ for any $Y$ and all  $k\ge 2$, 
including $Y$ being a random graph, as suggested by Defant and Kravitz, and  pose some open problems.

%Let $\mathsf{Lollipop}_{n-k,k}$ be a lollipop graph of order $n$ obtained by  identifying one end of a path of order $n-k+1$ with a vertex of a complete graph of order $k$. Defant and Kravitz started to study the connectedness of $\FS(\mathsf{Lollipop}_{n-k,k},Y)$. In this paper, we give a sufficient and necessary condition for $\FS(\mathsf{Lollipop}_{n-k,k},Y)$ to be connected  for all $2\leq k\leq n$. 
\end{abstract}

\maketitle
\section{Introduction}

 All graphs considered in this paper are finite and simple without loops. Let $G=(V(G),E(G))$ be a graph. For $W\subseteq V(G)$, $G|_W$ denotes the subgraph of $G$ induced by $W$ and $N(W)$ denotes the vertex set consisting of all vertices in $V(G)\setminus W$ that are adjacent to some vertex in $W$. Let $C_n$ denote a cycle of order $n$ and $K_{s,t}$  a complete bipartite graph with   bipartition of size $s$ and $t$. In particular, set $S_n=K_{1,n-1}$. By symmetry, when the graph $K_{s,t}$ is considered, we always assume that $s\le t$.  
An edge of a connected graph is a cut edge if its removal results a disconnected graph.   A cut edge is non-trivial if  none of  its ends has degree one. If a path $P=v_1v_2\cdots v_k$ is a $(k-2)$-subdivisions of a non-trivial cut edge $v_1v_k$ of a connected graph, then we call $P$ a non-trivial $k$-bridge of the resulting graph.   
%A non-trivial $k$-bridge in a graph $X$ is a path $a_1,\dots,a_k$ such that $a_{i+1}$ and $a_{i-1}$ are the only neighbors of $a_i$, for $i\in \{2, \dots, k-1\}$, and in $X|_{V(X)-\{a_2, \dots, a_{k-1}\}}$ the vertices $a_1$ and $a_k$ lie in different connected components, which both have size at least $2$. 
A non-trivial cut edge corresponds to a non-trivial $2$-bridge. For a finite sequence $S$, let $S^{-1}$ denote the reverse of $S$. 

The friends-and-strangers graphs, introduced by Defant and Kravitz \cite{DK}, is a kind of flip graphs defined as follows.

 \begin{defn}\cite{DK} 
Let $X$ and $Y$ be two graphs, each with $n$ vertices. The friends-and-strangers graph $\mathsf{FS}(X,Y)$ of $X$ and $Y$ is a graph with vertex set consisting of all bijections from $V(X)$ to $V(Y)$, two such bijections $\sigma$, $\sigma'$ are adjacent if and only if they differ precisely on two adjacent vertices, say $a,b\in V(X)$ with $ab\in E(X)$, and the corresponding mappings are adjacent in $Y$, i.e.,
\begin{itemize}
\item  $\sigma(a)\sigma(b) \in E(Y)$;
\item  $\sigma(a)=\sigma'(b)$, $\sigma(b)=\sigma'(a)$ and $\sigma(c)=\sigma'(c)$ for all $c\in V(X)\backslash \{a,b\}.$
\end{itemize}
\end{defn}

 When this is the case, we refer to the operation that transforms $\sigma$ into $\sigma'$ as an $(X,Y)$-{\it friendly swap}, and say that the  swap along the edge $ab\in E(X)$ transforms $\sigma$ to $\sigma'$. 
%Let $S$ be a sequence of swaps along edges $e_1,\dots,e_t$, we say ``apply $S$ to $\sigma$'' to mean that we swap $\sigma$ in order, which starts at $e_1$ and ends at $e_t$.

The friends-and-strangers graph $\mathsf{FS}(X,Y)$ can be interpreted as follows.  View $V(X)$ as $n$ cities and $V(Y)$ as $n$ mayors.  Two mayors are friends if and only if they are adjacent in $Y$ and two cities are adjacent if and only if they are adjacent in $X$. A bijection from $V(X)$ to $V(Y)$ represents $n$ mayors managing these $n$ cities such that each mayor manage precisely one city.
 At any point of time, two mayors can swap their cities if and only if they are friends and the two cities they manage are adjacent. A natural question is  how various configurations can be reached from other configurations when  multiple such swaps are allowed. This is precisely the information that is encoded in $\mathsf{FS}(X,Y)$. Note that the components of $\mathsf{FS}(X,Y)$ are the equivalence classes of mutually-reachable (by the multiple swaps described above) configurations, so the connectivity,  is the basic aspect  of interest in  friends-and-strangers graphs.

%The questions and results in literature on the friends-and-strangers graph $\mathsf{FS}(X,Y)$ roughly fall in three types
%\begin{itemize}
%\item The structure of $\mathsf{FS}(X,Y)$ when at least one of $X,Y$ are specific graphs, such as  stars, paths, cycles,  spider graphs and so on, \cite{DDLW}, \cite{DK}, \cite{J1}, \cite{L}, \cite{WC}, \cite{W}.  
%\item The structure  of $\FS(X,Y)$ when none of $X,Y$ is specific graph, such as minimum degree conditions on $X$ and $Y$, the case when $X$ has a Hamiltonian path, the non-polynomially bounded diameters and so on,  \cite{ADK}-\cite{DK}, \cite{J1}, \cite{J2}.
%\item The structure of $\mathsf{FS}(X,Y)$ when both $X$ and $Y$ are random graphs, \cite{ADK}, \cite{M}, \cite{Wang}.  \end{itemize}

The questions and results in literature on the friends-and-strangers graph $\mathsf{FS}(X,Y)$ can be divided into two categories. One is when at least one of $X,Y$ are specific graphs, such as  stars, paths, cycles,  spider graphs and so on, \cite{DDLW}, \cite{DK}, \cite{J1}, \cite{L}, \cite{WC}, \cite{W}.  
 Another is when none of $X,Y$ is specific graph, such as both $X$ and $Y$ are random graphs, minimum degree conditions on $X$ and $Y$, the non-polynomially bounded diameters of $\FS(X,Y)$ and so on,  \cite{ADK}-\cite{DK}, \cite{J1}, \cite{J2}, \cite{M}, \cite{Wang}. 
 We  note that Milojevic \cite{M} also studied a new model of friends-and-strangers graphs.

 The structure of $\mathsf{FS}(X,Y)$ when  $X,Y$ belong to the first category is a basic question on the topic related to friends-and-strangers graphs, and the results on this category can also be used to study the other category. For example, Alon, Defant and Kravitz \cite{ADK} used the structure of $\FS(S_n,Y)$ in  researching  the  threshold probability and minimum degree conditions on $X, Y$ for the connectedness of $\FS(X,Y)$;  Jeong \cite{J2} used the structure of 
 $\FS(C_n,Y)$ to investigate the connectedness of $\FS(X,Y)$ when $X$ is 2-connected.

 The first, also foundational, result in  the literature on friends-and-strangers graphs is the following, derived by Wilson \cite{W} using algebraic methods, which gives  a sufficient and necessary condition for $\FS(S_n, Y)$ to be connected. 
\begin{thm}\cite{W}\label{WW}
Let $Y$ be a graph on $n\ge3$ vertices. The graph $\FS(S_n,Y)$ is connected if and only if $Y$ is $2$-connected, non-bipartite and $Y\not=C_n, \varTheta$, where $\varTheta$ is a graph of order $7$ as shown in Figure 1. 
%\footnote{This picture is drawn by Defant and Kravitz \cite{DK}.}$.
\end{thm}

 \begin{figure}[h]
\centering
\includegraphics[scale=0.5]{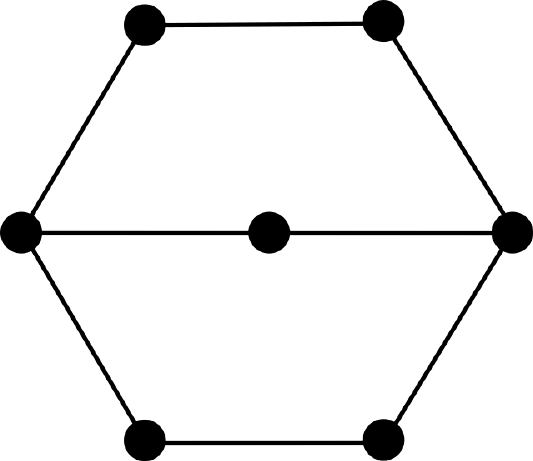}
\centering 
%\vspace{0.1cm}
\caption {The graph $\varTheta$.}
\end{figure}

\vskip 0cm
Note that an $S_n$ is also a $K_{1,n-1}$. Defant and Kravitz \cite{DK}  suggested to investigate the connectivity of $\mathsf{FS}(K_{k,n-k},Y)$ for any graph $Y$ and they thought it might be interesting even if to consider  the case when $k=2$.

 In this paper, by using combinatorial methods, we  consider the connectedness of $\mathsf{FS}(K_{k,n-k},Y)$ in general situation. At first, we get the following theorem that tells us when $\mathsf{FS}(K_{k,n-k},Y)$ is disconnected for $k\geq 2$, which is a little different from that in Theorem \ref{WW}.
 
 \begin{thm}\label{T}
Let $Y$ be a graph on $n\ge 2k\ge 4$ vertices.  If $Y$ is disconnected, or is bipartite, or contains a non-trivial $k$-bridge, or $Y=C_n$, then $\FS(K_{k,n-k},Y)$ is disconnected.
\end{thm}
Next, we consider when $\mathsf{FS}(K_{k,n-k},Y)$ is connected. For $k=2$, we characterize the connectedness of $\mathsf{FS}(K_{2,n-2},Y)$ completely as follows.
\begin{thm}\label{TT}
Let $Y$ be a graph on $n\ge 4$ vertices. Then the graph $\FS(K_{2,n-2},Y)$ is connected if and only if $Y$ is a connected non-bipartite graph with no non-trivial cut edge and $Y\not=C_n$.
\end{thm}
For $k\geq 3$, we fail in doing that as in Theorem \ref{TT}, and obtain a sufficient condition for $\mathsf{FS}(K_{k,n-k},Y)$ to be connected as below.

\begin{thm}\label{TTT}
Let $Y$ be a graph on $n\ge 2k\ge 6$ vertices. If  $Y$ is a $(k-1)$-connected, non-bipartite graph and $Y\not=C_n$,
then the graph $\FS(K_{k,n-k},Y)$ is connected. \end{thm}

 %In this paper, by using combinatorial methods, we  give a sufficient and necessary condition for $\mathsf{FS}(K_{2,n-2},Y)$ to be connected, and then we establish a sufficient condition for $\mathsf{FS}(K_{k,n-k},Y)$ to be connected. 

Finally, we determine the threshold probability that guarantees $\mathsf{FS}(K_{k,n-k}, Y)$ being connected when $Y$ is a random graph. Let $\mathcal{G}(n,p)$ denote the  Erd\H{o}s-R\'{e}nyi random graphs with $n$ vertices and edge-chosen probability $p$.
 
\begin{thm} \label{TTTT}
Let $Y$ be a random graph chosen from $\mathcal{G}(n,p)$. For any fix integer $k\ge 1$, the threshold probability guaranteeing the connectedness  of $\FS(K_{k,n-k},Y)$ is $$p_0=\big(1+o(1)\big)\frac{\log n} {n}.$$
\end{thm}

The remainder of this paper is organized as follows. Section \ref{s2} contains some preliminaries. Sections \ref{s3} and \ref{s4} are devoted to prove Theorems \ref{T}, \ref{TT} and Theorems \ref{TTT},  \ref{TTTT}, respectively. In Section \ref{s5}, we raise some open problems.

\section{Preliminaries}\label{s2}

In this section, we give some results for proving Theorems \ref{T}, \ref{TT}, \ref{TTT} and \ref{TTTT}. The first four results are the basic properties of friends-and-strangers graphs.
\begin{lem} \cite{DK} \label{dk} For any two graphs $X$ and $Y$ on $n$ vertices,
the graphs $\mathsf{FS}(X,Y)$ and $ \mathsf{FS}(Y,X)$ are isomorphic.\end{lem}
\begin{lem} \cite{DK} \label{asd}Let $X,\widetilde{X},Y,\widetilde{Y}$ be graphs on $n$ vertices. If $X, Y$ are  spanning subgraphs of $\widetilde{X}, \widetilde{Y}$, respectively, then $\mathsf{FS}(X,Y)$ is a spanning subgraph of $\mathsf{FS}(\widetilde{X},\widetilde{Y})$. In particular, $\mathsf{FS}(\widetilde{X},\widetilde{Y})$ is connected if $\mathsf{FS}(X,Y)$ is connected. 
\end{lem}

\begin{lem}\cite{DK}\label{20}
 Let $X$ and $Y$ be two graphs on $n$ vertices. If one of $X$ and $Y$ is disconnected, then the graph $\FS(X,Y)$ is disconnected.
\end{lem}

\begin{lem}\cite{DK}\label{21}
 If both $X$ and $Y$ are bipartite graphs on $n$ vertices, then the graph $\FS(X,Y)$ is disconnected.
\end{lem}

 The following lemma will be used to investigate the disconnectedness of friends-and-strangers graphs, due to Milojevic \cite{M}.

\begin{lem}\cite{M}\label{22}
Let $X$ be a graph containing a non-trivial $k$-bridge. If $Y$ is not $(k+1)$-connected, then the graph $\FS(X,Y)$ is disconnected.
\end{lem}

 The following lemma is an addition to Theorem \ref{WW}, due to Defant and Kravitz \cite{DK}. Let $S_n^+$ denote a graph obtained from $S_n$ by adding one extra edge.

\begin{lem}\cite{DK}\label{31}
Let $Y$ be a graph on $n\ge 4$ vertices. Then the graph $\FS(S_n^+,Y)$ is connected if  $Y$ is $2$-connected and $Y\not=\varTheta$, $C_n$.
\end{lem}

 Defant and Kravitz \cite{DK} charactered the components of $\FS(C_n,S_n)$. Note that an $S_n$ is also a $K_{1,n-1}$,  the following lemma extends their result and reveals the structure of the graph $\FS(C_n,K_{k,n-k})$.

\begin{lem}\label{23} Let $A$ and $B$ be the bipartition of the vertex set of a $K_{k,n-k}$. 
 Then $\FS(C_n,K_{k,n-k})$ has precisely $(k-1)!(n-k-1)!$ components, each of which corresponds to a pair of cyclic orderings of $A$ and $B$.
\end{lem}
\begin{proof}[\bfseries{Proof}]
  Assume that $c_1c_2\cdots c_nc_1$ is a $C_n$. For any bijection $\sigma: V(C_n)\mapsto V(K_{k,n-k})$, we can map $\sigma$  to 
a cyclic ordering $\rho(\sigma)=\sigma(c_1)\sigma(c_2)\cdots \sigma(c_n)\sigma(c_1)$ of $V(K_{k,n-k})$. This cyclic ordering $\rho(\sigma)$ can be divided uniquely to a pair of cyclic orderings $(\rho_A(\sigma),\rho_B(\sigma))$ of $A$ and $B$. So we obtain a map that sends any bijection $\sigma$ to a pair of cyclic orderings $(\rho_A(\sigma),\rho_B(\sigma))$ of $A$ and $B$. 

 Because any vertices $a\in A$ and $b\in B$ are adjacent in $K_{k,n-k}$, two bijections are in the same component of $\FS(C_n,K_{k,n-k})$ if and only if they are mapped to the same pair of cyclic orderings. Thus, each component of $\FS(C_n,K_{k,n-k})$ corresponds  to a pair of cyclic orderings of $A$ and $B$. Because a set $S$ has precisely
 $|S|! / |S|=(|S|-1)!$ cyclic orderings, the graph
$\FS(C_n,K_{k,n-k})$ has precisely $(k-1)!(n-k-1)!$ components.
\end{proof}

For any graph $G$ and $u,v\in V(G)$, we use $(u \ v)$ to denote the bijection  $V(G) \mapsto V(G)$ such that $(u \ v) (u)=v$, $(u \ v) (v)=u$ and $(u \ v) (w)=w$ for any $w\in V(G)\backslash \{u,v\}$.
In order to study the connectedness of friends-and-strangers graphs,  Alon, Defant and Kravitz \cite{ADK} introduced the notion of an exchangeable pair of vertices:
 Let $X$ and $Y$ be two graphs on $n$ vertices,   $\sigma:V(X)\mapsto V(Y)$ be a bijection and 
 $u,v \in V(Y)$. 
 We say that $u$ and $v$ are $(X,Y)${\it-exchangeable from $\sigma$} if $\sigma$ and $(u \ v) \circ \sigma$, i.e., $\sigma\circ (\sigma^{-1}(u)\ \sigma^{-1}(v))$, are in the same component. In other words,  we say $u$ and $v$ are $(X,Y)$-exchangeable from $\sigma$ if there  is  a  sequence of
$(X,Y)$-friendly swaps that we can apply to $\sigma$ in order to exchange $u$ and $v$, that is, there is a  path between $\sigma$ and $(u \ v) \circ \sigma$ in $\mathsf{FS}(X,Y)$. 
   The following   two lemmas give  two sufficient conditions for $\mathsf{FS}(X,Y)$ to be connected in terms of exchangeable pairs of vertices.

\begin{lem}\cite{ADK}\label{lem:exchangeable}
Let $X$, $Y$ and $\widetilde Y$ be three graphs on $n$ vertices such that $Y$ is a spanning subgraph of $\widetilde Y$.  Suppose that for any edge $uv\in E(\widetilde Y)$ and any bijection $\sigma:V(X)\mapsto V(Y)$ satisfying $\sigma^{-1}(u) \sigma^{-1}(v) \in E(X)$, the vertices $u$ and $v$ are $(X,Y)$-exchangeable from $\sigma$. Then the number of components of $\FS(X,Y)$ equals to the number of components of $\FS(X,\widetilde Y)$. In particular, the graph  $\FS(X,Y)$ is connected if and only if $\FS(X,\widetilde Y)$ is connected.
\end{lem}

Godsil and Royle \cite{GR} proved that the graph $\FS(X,K_n)$ is connected if and only if $X$ is connected.  By setting $\widetilde Y=K_n$ in Lemma \ref{lem:exchangeable}, the following lemma holds. 

\begin{lem} \cite{ADK} \label{ex}
 Let $X, Y$ be two graphs on $n$ vertices such that  $X$ is connected. Suppose that for any two vertices $u,v\in V(Y)$ and every $\sigma$ satisfying $\sigma^{-1}(u)\sigma^{-1}(v) \in E(X)$, the vertices $u$ and $v$ are $(X,Y)$-exchangeable from $\sigma$. Then $\mathsf{FS}(X,Y)$ is connected.
\end{lem}

 Throughout the rest part of this paper, we always assume that $K_{k,n-k}$ has vertex set $[n]=\{1,2,...,n\}$ with  bipartition $\{1,\dots, k\}$ and $\{k+1,\dots, n\}$.

The following lemma  reveals the exchageability of a pair of vertices in $K_{2,n-2}$.

\begin{lem}\label{l1}
Suppose that  $t\ge 3$ is an odd integer. Then the vertices $1, 2\in V(K_{2,t-2})$ are $(C_t, K_{2,t-2})$-exchangeable from any bijection $\sigma:V(C_t)\mapsto V(K_{2,t-2})$ satisfying $\sigma^{-1}(1)\sigma^{-1}(2)\in E(C_t)$. 
\end{lem}
\begin{proof} [\bfseries{Proof}]  Fix such a bijection $\sigma$ and let $C_t=c_1c_2\cdots c_tc_1$ satisfying $c_{1}=\sigma^{-1}(1), c_2=\sigma^{-1}(2)$. Let $S$ be the sequence of swaps along the edges  $c_2 c_3, c_3c_4,\dots,$ $c_{t-1} c_{t},$ $c_{1}c_2, c_{t}c_{1}$. Then the sequence $S$ transforms $\sigma$ into a new bijection $\sigma_1=S(\sigma)$ satisfying $\sigma_1(c_{2})=\sigma(c_{1}),$ $\sigma_1(c_{1})=\sigma(c_2),$ $ \sigma_1(c_{3})=\sigma(c_4),$ $\sigma_1(c_{4})=\sigma(c_{5}),$ $\dots,\sigma_1(c_{t-1})=\sigma(c_t),\sigma_1(c_t)=\sigma(c_3)$. Set $\sigma_{i+1}=S(\sigma_i)$, see Table 1. It is easy to verify that $\sigma_{t-2}=(1 \ 2)\circ \sigma$ since $t-2$ is an odd integer,  which implies that $\sigma$ and $(1 \ 2)\circ \sigma$ lie in the same component of $\FS(C_t,K_{2,t-2})$.
\end{proof}

\begin{center}
\begin{tabular*}   {10.35cm}{@{\hspace{0mm}}c @{\hspace{2.25mm}}| @{\hspace{3.25mm}}c @{\hspace{3.25mm}} @{\hspace{2.25mm}}c @{\hspace{2.25mm}} | @{\hspace{2.25mm}}c @{\hspace{2.25mm}}  @{\hspace{2.25mm}}c @{\hspace{2.25mm}}  @{\hspace{2.25mm}}c @{\hspace{2.25mm}}   @{\hspace{2.25mm}}c @{\hspace{2.25mm}} @{\hspace{2.25mm}}c @{\hspace{2.25mm}} @{\hspace{2.25mm}}c @{\hspace{2.25mm}} @{\hspace{2.25mm}}c @{\hspace{2.25mm}}  @{\hspace{2.25mm}}c @{\hspace{2.25mm}} @{\hspace{2.25mm}}c @{\hspace{2.25mm}}  @{\hspace{2.25mm}}c @{\hspace{2.25mm}}   @{\hspace{2.25mm}}c @{\hspace{2.25mm}} @{\hspace{2.25mm}}c @{\hspace{2.25mm}} @{\hspace{2.25mm}}c @{\hspace{2.25mm}}  @{\hspace{2.25mm}}c @{\hspace{2.25mm}}   @{\hspace{2.25mm}}c @{\hspace{2.25mm}} }

 \hline
   $~\empty$ &$c_1$&$c_2$&$c_3$&$c_4$&$\cdots$&$c_{t-1}$&$c_t$\\
\hline
   $~ \sigma$ &$\sigma(c_1)$&$\sigma(c_2)$&$\sigma(c_3)$&$\sigma(c_4)$&$\cdots$&$\sigma(c_{t-1})$&$\sigma(c_t)$\\

   $~\sigma_1$ &$\sigma(c_2)$&$\sigma(c_1)$&$\sigma(c_4)$&$\sigma(c_5)$ &$\cdots$&  $\sigma(c_t)$&$\sigma(c_3)$\\
 
   $~\sigma_2$ &$\sigma(c_1)$&$\sigma(c_2)$&$\sigma(c_5)$&$\sigma(c_6)$&$ \cdots $&$\sigma(c_3)$&$\sigma(c_4)$\\

   $~\sigma_{t-2}$ &$\sigma(c_1)$&$\sigma(c_2)$&$\sigma(c_3)$&$\sigma(c_4)$&$\cdots$&$\sigma(c_{t-1})$&$\sigma(c_t)$\\
   
\hline
\end{tabular*}\\
\begin{center}  Table 1. The images of $V(C_t) $ under $\sigma$, $\sigma_1$, $\sigma_2$ and $\sigma_{t-2}$.   \end{center}
\end{center}

 %We will make use of Lemmas \ref{ex} and  \ref{lem:exchangeable} to prove the connectedness part of Theorems \ref{TT} and \ref{TTT}, respectively,   by verifying the exchangeability of any two vertices. 
 The following lemma is due to Bangachev \cite{B}.

\begin{lem}\cite{B}\label{ex2}
 If two bijections $\sigma, \tau\in V(\FS(X,Y))$ are in the same component of $\FS(X,Y)$ and there is a sequence of swaps transforming $\sigma$ into $\tau$ only involving edges in $Y|_{V(Y)\setminus \{u,v\}}$, then $u,v$ are $(X,Y)$-exchangeable from $\tau$ if and only if $u,v$ are  $(X,Y)$-exchangeable from $\sigma$.
\end{lem}

\section{Proof of Theorems \ref{T} and \ref{TT}}\label{s3}

\begin{proof}[\bfseries{Proof of Theorem \ref{T}}]
 If $Y$ is disconnected, Lemma \ref{20} implies that $\FS(K_{k,n-k},Y)$ is disconnected. If $Y$ is bipartite, Lemma \ref{21} implies that $\FS(K_{k,n-k},Y)$ is disconnected since $K_{k,n-k}$ is also bipartite. 
If $Y$ contains a non-trivial $k$-bridge, then Lemma \ref{22} implies that $\FS(K_{k,n-k},Y)$ is disconnected since $K_{k,n-k}$ is not $(k+1)$-connected. If $Y$ is  a  cycle, then Lemma \ref{23} implies that $\FS(K_{k,n-k},Y)$ is disconnected.
\end{proof}\vspace{-0.25cm}

\begin{proof}[\bfseries{Proof of Theorem \ref{TT}}] It is easy to see that  
the disconnectedness part  of Theorem \ref{TT} follows from Theorem \ref{T}, just taking $k=2$.

 For the connectedness part of Theorem \ref{TT}, we need the following.
 \begin{prop}\label{P2}
Suppose that $X$ is a connected non-bipartite graph on $n\ge 5$ vertices with no non-trivial cut edge and $X\not=C_n$. Then the graph $\FS(X, K_{2,n-2})$ is connected. 
\end{prop}
To make the arguments easier to follow, we postpone the proof of Proposition \ref{P2} until the end of this section.
 Note that  $K_{k,n-k}=C_4$ if $n=4$ and so Theorem \ref{TT} holds by the connectedness of $\FS(C_n,Y)$  as shown by Defant and Kravitz \cite{DK}. If  $n\ge 5$, then Theorem \ref{TT} follows from Proposition \ref{P2}. 
 
 Therefore, we complete the proof of Theorem \ref{TT}.
 \end{proof}

 We are now in position to prove Proposition \ref{P2}. Before starting to do this, we
 need in addition the following lemma. Recall that the bipartition of $K_{2,n-2}$ is  $\{1, 2\}$ and $\{3,\dots,n\}$.
\begin{lem}\label{l2}
 Suppose that $X$ is a connected non-bipartite graph on $n\ge 5$ vertices. Then the vertices $1, 2\in V(K_{2,n-2})$ are $(X, K_{2,n-2})$-exchangeable from any bijection $\sigma:V(X)\mapsto V(K_{2,n-2})$ satisfying $\sigma^{-1}(1)\sigma^{-1}(2)\in E(X)$. 
\end{lem}
\begin{proof}[\bfseries{Proof}] Since $X$ is non-bipartite, there exists an odd cycle $C$ in $X$. 
Fix a bijection $\sigma$ satisfying $\sigma^{-1}(1)\sigma^{-1}(2)\in E(X)$. 

If both $\sigma^{-1}(1)$ and $\sigma^{-1}(2)$ are in $V(C)$, then  Lemma \ref{l1} implies directly that $1,2 \in V(K_{2,n-2})$ are $(X, K_{2,n-2})$-exchangeable from $\sigma$ by considering swaps along edges in $E(C)$.

If exactly one of $\sigma^{-1}(1)$ and $\sigma^{-1}(2)$ is in $V(C)$, assume  that $\sigma^{-1}(1)\in V(C)$. Let $c\in N(\sigma^{-1}(1))\cap V(C)$ be one of the two neighbors of $\sigma^{-1}(1)$ in $V(C)$ and swap along the edges $\sigma^{-1}(1)c, \sigma^{-1}(2)c$, which results a new bijection $\tau$ satisfying $\tau^{-1}(1),\tau^{-1}(2)\in V(C)$ and $\tau^{-1}(1) \tau^{-1}(2)\in E(C)$. Applying Lemma \ref{l1} to $\tau$ obtains a sequence of swaps that transforms $\tau$ to $(1 \ 2)\circ \tau$. Then we can swap along the edges $\sigma^{-1}(2)c, \sigma^{-1}(1)c$ to get the desired $(1 \ 2)\circ \sigma$.

If none of $\sigma^{-1}(1)$ and $\sigma^{-1}(2)$ are  in $V(C)$, let $x_1x_2\cdots x_p$ be a shortest path in $X$ between the two vertex sets $\{\sigma^{-1}(1),\sigma^{-1}(2)\}$ and $V(C)$. Assume without loss of generality that $\sigma^{-1}(1)=x_1$ and denote $\sigma^{-1}(2)$ by $x_0$. Denote  by $S$ the sequence of swaps  along the edges $x_1x_2,\dots, x_{p-1}x_p$, $x_0 x_1$, $x_1x_2,\dots, x_{p-2}x_{p-1}$. The sequence $S$ transforms $\sigma$ into a new bijection $\tau$ satisfying $\tau^{-1}(1)\in V(C)$, $\tau^{-1}(2)\notin V(C)$ and $\tau^{-1}(1)\tau^{-1}(2)\in E(X)$. Similar to the case when  exactly one of $\sigma^{-1}(1)$ and $\sigma^{-1}(2)$ is in $V(C)$, there is a sequence of swaps that transforms $\tau$ into $(1 \ 2)\circ \tau$. Then, the sequence  $S^{-1}$ transforms $(1 \ 2)\circ \tau$ into the desired $(1 \ 2)\circ \sigma$.
\end{proof}

%\noindent{\bf Remark.} The main obstruction in charactering the connectedness of $\FS(X,K_{k,n-k})$ for $k\ge 3$ is that we failed to obtain a  analogue of Lemma \ref{l2}.

\begin{proof}[\bfseries{Proof of Proposition \ref{P2}}]
 By Lemma \ref{lem:exchangeable}, it suffices to show that any two vertices $u, v\in V(K_{2,n-2})$ are $(X, K_{2,n-2})$-exchangeable from any bijection $\sigma:V(X)\mapsto V(K_{2,n-2})$ satisfying $\sigma^{-1}(u)\sigma^{-1}(v)\in E(X)$. If $\{u,v\}=\{1,2\}$, then Lemma \ref{l2} implies directly that $u, v$ are $(X, K_{2,n-2})$-exchangeable from $\sigma$. If $|\{u,v\}\cap\{1,2\}|=1$, then we can swap along the edge $\sigma^{-1}(u)\sigma^{-1}(v)$ to obtain directly the bijection $(u \ v)\circ \sigma$, i.e., $u, v$ are $(X, K_{2,n-2})$-exchangeable from $\sigma$. So we are left to consider  the case $\{u,v\}\cap \{1,2\}=\varnothing$.  
 
 Firstly, we prove that there exists a sequence of swaps only involving edges in $K_{2,n-2}|_{[n]\setminus \{u,v\}}$ that transforms $\sigma$ into a bijection $\tau$ such that  $X|_{\tau^{-1}(\{u,v,1,2\})}$ is connected and $\tau^{-1}(u)\tau^{-1}(v)\in E(X)$. If there are no edges between $\{\sigma^{-1}(1), \sigma^{-1}(2)\}$ and $\{\sigma^{-1}(u), \sigma^{-1}(v)\}$, let $x_1x_2\cdots x_q$ be a shortest path between the two vertex sets $\{\sigma^{-1}(1),\sigma^{-1}(2)\}$ and $\{\sigma^{-1}(u),\sigma^{-1}(v)\}$. Assume without loss of generality that $\sigma^{-1}(1)=x_1$. Denote by $S$ the sequence of swaps along the edges $x_1x_2,\dots, x_{q-2}x_{q-1}$, which can transform $\sigma$ into a bijection $\tau'$ such that $X|_{\tau'^{-1}(\{u,v,1\})}$ is connected and $\tau'^{-1}(u)\tau'^{-1}(v)\in E(X)$. For the same reason, if $\tau'^{-1}(2)$ is not adjacent to any vertex in $\{ \tau'^{-1}(u),\tau'^{-1}(v)$, $\tau'^{-1}(1)\}$, then there exists a sequence of swaps  transforming $\tau'$ into a bijection $\tau$ satisfying that $X|_{\tau^{-1}(\{u,v,1,2\})}$ is connected and $\tau^{-1}(u)\tau^{-1}(v)\in E(X)$.
 So, by Lemma \ref{ex2}, $u,v$ are $(X,K_{2,n-2})$-exchangeable from $\tau$ if and only if $u,v$ are  $(X,Y)$-exchangeable from $\sigma$.  So, we assume that $X|_{\sigma^{-1}(\{u,v,1,2\})}$ is connected. 

If one of $\sigma^{-1}(1)$ and $\sigma^{-1}(2)$ has no neighbor in $\{\sigma^{-1}(u),\sigma^{-1}(v)\}$, assume that it is $\sigma^{-1}(1)$ and  that $\sigma^{-1}(2)\sigma^{-1}(v)\in E(X)$. Denote by $S_0$ the sequence of swaps along the edges $\sigma^{-1}(2)\sigma^{-1}(v)$, $\sigma^{-1}(v)\sigma^{-1}(u)$. The sequence  $S_0$ transforms $\sigma$ into a new bijection $\eta$ such that both $\eta^{-1}(1)$ and $\eta^{-1}(2)$ have neighbors in $\{\eta^{-1}(u),\eta^{-1}(v)\}$ and $\eta^{-1}(u)\eta^{-1}(v)\in E(X)$. If $u,v$ are $(X,K_{2,n-2})$-exchangeable from $\eta$, then there will exist a sequence $S_1$ of swaps  that transforms $\eta$ to $(u \ v)\circ \eta$. Thus, the sequences $S_0,S_1,S^{-1}_0$ will transform $\sigma$ to $(u \ v)\circ \sigma$.  So, to complete the proof, it suffices to show that $u,v$ are $(X,K_{2,n-2})$-exchangeable from $\sigma$ under the assumption that both $\sigma^{-1}(1)$ and $\sigma^{-1}(2)$ have neighbors in $\{\sigma^{-1}(u),\sigma^{-1}(v)\}$. 

\vskip 2mm
\noindent\textbf{Case 1.}  The vertices $\sigma^{-1}(1)$ and $\sigma^{-1}(2)$ have a common neighbor in $\{\sigma^{-1}(u),\sigma^{-1}(v)\}$.
\vskip 2mm

Assume that $\sigma^{-1}(u)$ is a common neighbor. Swapping along the edges $\sigma^{-1}(1)\sigma^{-1}(u)$, $\sigma^{-1}(u)\sigma^{-1}(v)$, $\sigma^{-1}(2)\sigma^{-1}(u)$ transforms $\sigma$ into a new bijection $\tau$ with $\tau^{-1}(2)=\sigma^{-1}(u)$, $\tau^{-1}(1)=\sigma^{-1}(v)$ and $\tau^{-1}(1)\tau^{-1}(2)\in E(X)$, see Figure 2. Apply Lemma \ref{l2} to $\tau$, we obtain a sequence of swaps that transforms $\tau$ to $(1\ 2)\circ \tau$. Then swapping along the edges $\sigma^{-1}(1)\sigma^{-1}(u)$, $\sigma^{-1}(v)\sigma^{-1}(u)$, $\sigma^{-1}(2)\sigma^{-1}(u)$ transforms $(1\ 2)\circ \tau$ into the desired bijection $(u\ v)\circ \sigma$.
\vskip 3mm
 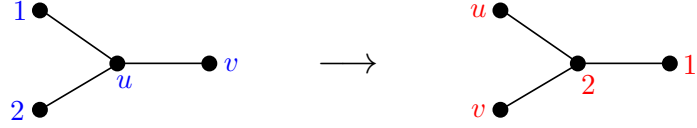
\begin{figure}[htb]\label{f22}
\centering
\includegraphics[scale=0.83]{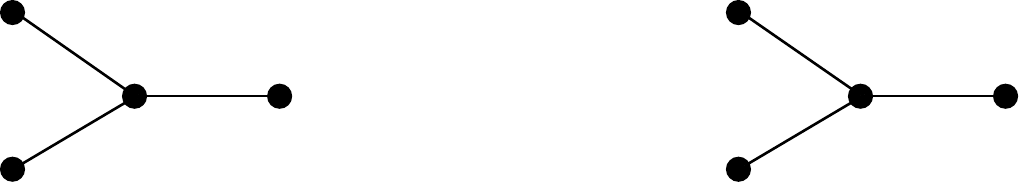}
\centering 
\vspace{0.1cm}
\put(-130,18.2){\makebox(3,3){{\color{black}{\Large{$\longrightarrow
$}}}}}
\put(-254,39){\makebox(3,3){{\color{blue}{$1$}}}}
\put(-255,1){\makebox(3,3){{\color{blue}{$2$}}}}
\put(-80.5,39){\makebox(3,3){{\color{red}{$u$}}}}
\put(-80.5,1){\makebox(3,3){{\color{red}{$v$}}}}
\put(-214.5,12){\makebox(3,3){{\color{blue}{$u$}}}}
\put(-39,11){\makebox(3,3){{\color{red}{$2$}}}}
\put(-174,18.5){\makebox(3,3){{\color{blue}{$v$}}}}
\put(-0.5,18.7){\makebox(3,3){{\color{red}{$1$}}}}
\caption {The blue {\color{blue}{$i$}} is $\sigma^{-1}(i)$ and red {\color{red}{$i$}}  is $\tau^{-1}(i)$.}
\end{figure}
\vskip 2mm
\noindent\textbf{Case 2.} The vertices  $\sigma^{-1}(1)$ and $\sigma^{-1}(2)$ have no common neighbors in $\{\sigma^{-1}(u),\sigma^{-1}(v)\}$.
\vskip 2mm

Since $\sigma^{-1}(u)\sigma^{-1}(v)$ is not a trivial cut edge and $X$ contains no non-trivial  cut edge, there exist cycles in $X$ that contain the edge $\sigma^{-1}(u)\sigma^{-1}(v)$. Let $C$ be such a shortest one. The length of $C$ is strictly less than $n$ since $X$ is not a cycle. Assume without loss of generality that $\sigma^{-1}(1)\sigma^{-1}(u), \sigma^{-1}(2)\sigma^{-1}(v)$ are edges in $X$.

\vskip 2mm
\noindent\textbf{Subcase 2.1.} 
 Exactly one of $\sigma^{-1}(1)$ and $\sigma^{-1}(2)$ is in $V(C)$.
\vskip 2mm

 Assume that it is $\sigma^{-1}(1)$. Let $C= c_1c_2\cdots c_rc_1$ with $c_{1}=\sigma^{-1}(u)$, $c_2=\sigma^{-1}(1)$, $c_r=\sigma^{-1}(v)$.  

 Denote by $S$ the sequence of swaps along the edges
$c_2 c_3$, $c_3 c_4, \dots , c_{r-2} c_{r-1}$. The sequence  $S$ transforms $\sigma$ into a new bijection $\tau$ such that  $\tau^{-1}(1)$ and $\tau^{-1}(2)$ are adjacent to a common vertex $\tau^{-1}(v)$, see Figure $3$. The same as Case $1$,  we obtain a sequence of swaps that transforms $\tau$ to  $(u\ v)\circ \tau$. Then the sequence of swaps $S^{-1}$ transforms $(u\ v)\circ \tau$ into the desired $(u\ v)\circ \sigma$.

 \begin{figure}[htb]\label{f3}
\centering
\includegraphics[scale=0.59]{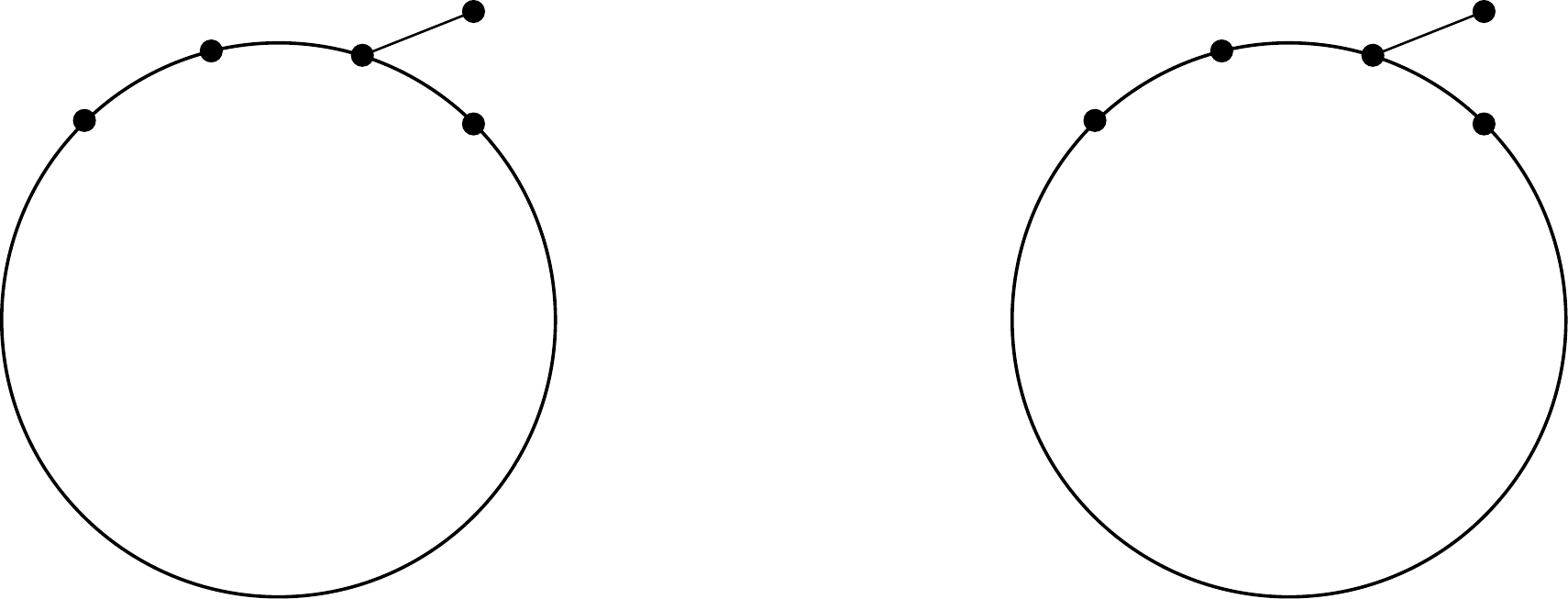}
\put(-146,49){\makebox(3,3){{\color{black}{\Large{$\longrightarrow
$}}}}}
\put(-250,93){\makebox(3,3){{\color{black}{$c_1$}}}}
\put(-271,81){\makebox(3,3){{\color{black}{$c_2$}}}}
\put(-228,92.7){\makebox(3,3){{\color{black}{$c_r$}}}}
\put(-210,79.2){\makebox(3,3){{\color{black}{$c_{r-1}$}}}}
\put(-62,93){\makebox(3,3){{\color{black}{$c_1$}}}}
\put(-83,81){\makebox(3,3){{\color{black}{$c_2$}}}}
\put(-40,92.7){\makebox(3,3){{\color{black}{$c_r$}}}}
\put(-22,79.2){\makebox(3,3){{\color{black}{$c_{r-1}$}}}}
\put(-284,91){\makebox(3,3){{\color{blue}{$1$}}}}
\put(-256,107){\makebox(3,3){{\color{blue}{$u$}}}}
\put(-225,106){\makebox(3,3){{\color{blue}{$v$}}}}
\put(-198.5,108){\makebox(3,3){{\color{blue}{$2$}}}}
\put(-183.5,90){\makebox(3,3){{\color{blue}{$\sigma(c_{r-1})$}}}}
\put(-105.5,91){\makebox(3,3){{\color{red}{$\sigma(c_3)$}}}}
\put(-69.5,107){\makebox(3,3){{\color{red}{$u$}}}}
\put(-37,106){\makebox(3,3){{\color{red}{$v$}}}}
\put(-11,108){\makebox(3,3){{\color{red}{$2$}}}}
\put(-11,90){\makebox(3,3){{\color{red}{$1$}}}}
\centering 
\vspace{0.1cm}
\caption {The blue {\color{blue}{$i$}} is $\sigma^{-1}(i)$ and red {\color{red}{$i$}}  is $\tau^{-1}(i)$.}
\end{figure}
\FloatBarrier

\vskip 2mm
\noindent\textbf{Subcase 2.2.}  None of $\sigma^{-1}(1)$ and $\sigma^{-1}(2)$ are in $V(C)$. \vskip 2mm

Let $a$ be the other neighbor of $\sigma^{-1}(u)$ in $V(C)$. 
Denote by $S$ the sequence of swaps along the edges
$\sigma^{-1}(2)\sigma^{-1}(v)$, $\sigma^{-1}(u)\sigma^{-1}(v)$, $\sigma^{-1}(u)a$, $\sigma^{-1}(1)\sigma^{-1}(u)$, $\sigma^{-1}(u)\sigma^{-1}(v)$, $\sigma^{-1}(2)\sigma^{-1}(v)$. The sequence $S$ transforms $\sigma$ into a new bijection $\tau$ satisfying $\tau^{-1}(2)\in V(C)$, $\tau^{-1}(1)\not\in V(C)$ and $\tau^{-1}(u)\tau^{-1}(v) \in E(X)$, see Figure $4$. By Subcase 2.1, we obtain a sequence of swaps that transforms $\tau$ to $(u\ v)\circ \tau$. 
Then the sequence  $S^{-1}$ transforms $(u\ v)\circ \tau$ into the desired $(u\ v)\circ \sigma$.

\vskip 3mm
 \begin{figure}[htb]\label{f5}
\centering
\includegraphics[scale=0.3]{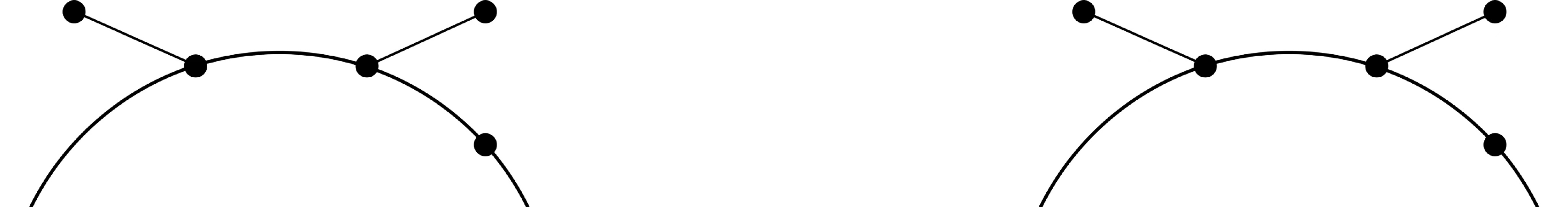}
\put(-150,14){\makebox(3,3){{\color{black}{\Large{$\longrightarrow
$}}}}}
\put(-192.5,13){\makebox(3,3){{\color{blue}{$\sigma(a)$}}}}
\put(-290,35.5){\makebox(3,3){{\color{blue}{$2$}}}}
\put(-37,31.5){\makebox(3,3){{\color{red}{$u$}}}}
\put(-70,31.5){\makebox(3,3){{\color{red}{$v$}}}}
\put(-228,31.5){\makebox(3,3){{\color{blue}{$u$}}}}
\put(-9,13){\makebox(3,3){{\color{red}{$2$}}}}
\put(-260,31.5){\makebox(3,3){{\color{blue}{$v$}}}}
\put(-98.8,35.5){\makebox(3,3){{\color{red}{$1$}}}}
\put(-211.2,5.8){\makebox(3,3){{\color{black}{$a$}}}}
\put(-20.3,5.8){\makebox(3,3){{\color{black}{$a$}}}}
\put(-199.7,35.5){\makebox(3,3){{\color{blue}{$1$}}}}
\put(-0.5,35.5){\makebox(3,3){{\color{red}{$\sigma(a)$}}}}

\centering 
\vspace{0.1cm}
\caption {The blue {\color{blue}{$i$}} is $\sigma^{-1}(i)$ and red {\color{red}{$i$}}  is $\tau^{-1}(i)$.}
\end{figure}

\vskip 2mm
\noindent\textbf{Subcase 2.3.} Both $\sigma^{-1}(1)$ and $\sigma^{-1}(2)$ are in $V(C)$.  \vskip 2mm

 Let $c_0$ be a vertex in $N(V(C))$. By the structure of $\FS(C_t,K_{2,t-2})$ described in Lemma \ref{23}, we may assume without loss of generality that $\sigma^{-1}(u)c_0 \in E(X)$.

 Denote by $S$ the sequence of swaps along the edges $\sigma^{-1}(2)\sigma^{-1}(v)$, $\sigma^{-1}(u)\sigma^{-1}(v)$, $\sigma^{-1}(u)c_0$, $\sigma^{-1}(1)\sigma^{-1}(u)$, $\sigma^{-1}(u)\sigma^{-1}(v)$, $\sigma^{-1}(2)\sigma^{-1}(v)$. The sequence  $S$ transforms $\sigma$ into a new bijection $\tau$ satisfying $\tau^{-1}(2)=c_0$, $\tau^{-1}(u)=\sigma^{-1}(u)$, $\tau^{-1}(v)=\sigma^{-1}(v)$ and $\tau^{-1}(1)=\sigma^{-1}(2)$, see Figure $5$. By Subcase 2.2, we obtain a sequence of swaps that transforms $\tau$ to $(u\ v)\circ \tau$. Then the sequence  $S^{-1}$ transforms $(u\ v)\circ \tau$ into the desired $(u\ v)\circ \sigma$.
\end{proof}
\vskip 2mm
\begin{figure}[htb]\label{f4}
\centering
\includegraphics[scale=0.3]{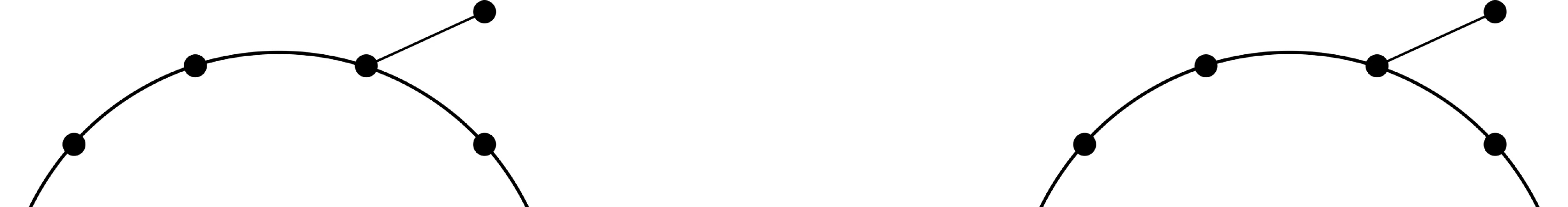}
\put(-150,14){\makebox(3,3){{\color{black}{\Large{$\longrightarrow
$}}}}}
\put(-200,12.5){\makebox(3,3){{\color{blue}{$1$}}}}
\put(-290,13){\makebox(3,3){{\color{blue}{$2$}}}}
\put(-38,31.5){\makebox(3,3){{\color{red}{$u$}}}}
\put(-71.5,31.5){\makebox(3,3){{\color{red}{$v$}}}}
\put(-229,31.5){\makebox(3,3){{\color{blue}{$u$}}}}
\put(-9,36){\makebox(3,3){{\color{red}{$2$}}}}
\put(-262,31.5){\makebox(3,3){{\color{blue}{$v$}}}}
\put(-99,13){\makebox(3,3){{\color{red}{$1$}}}}
\put(-205.5,27.5){\makebox(3,3){{\color{black}{$c_0$}}}}
\put(-14,27.5){\makebox(3,3){{\color{black}{$c_0$}}}}
\put(-190,36){\makebox(3,3){{\color{blue}{$\sigma(c_0)$}}}}
\put(-0,13.5){\makebox(3,3){{\color{red}{$\sigma(c_0)$}}}}
\centering 
\vspace{0.1cm}
\caption {The blue {\color{blue}{$i$}} is $\sigma^{-1}(i)$ and red {\color{red}{$i$}}  is $\tau^{-1}(i)$.}
\end{figure}

\section{Proofs of Theorems \ref{TTT} and  \ref{TTTT}}\label{s4}

In order to prove  Theorem \ref{TTT}, we need in addition the following lemma.
\begin{lem}\label{LL}
 Let $X$ be a $(k-1)$-connected non-bipartite graph on $n\ge 2k\ge 6$ vertices and $X \not=C_n$. Then for any vertices $u \in \{2,\dots,k\}$, $1,u \in V(K_{k,n-k})$ are $(X,K_{k,n-k})$-exchangeable from any bijection $\sigma:V(X)\mapsto V(K_{k,n-k})$ satisfying $\sigma^{-1}(u)\sigma^{-1}(1)\in E(X)$. 
\end{lem}
\begin{proof}[\bfseries{Proof}]
Assume without loss of generality that $u=2$.
Fix such a bijection $\sigma$ and let $C$ denote a shortest odd cycle in $X$,  whose existence is because $X$ being non-bipartite. We divide $\sigma^{-1}(\{1,\dots,k\})$ into two parts according to the cycle $C$, that is, we denote $L=L(\sigma)=V(C)\cap \sigma^{-1}(\{1,\dots,k\})$ and $M=M(\sigma)=\sigma^{-1}(\{1,\dots,k\}) \setminus L$. 

\vskip 2mm
\noindent\textbf{Case 1.} Both  $\sigma^{-1}(1)$ and $\sigma^{-1}(2)$ are in $V(C)$. 
\vskip 2mm

 We use induction on $|L|$. If $|L|=2$, then Lemma \ref{l1} implies directly that $1,2 \in V(K_{k,n-k})$ are $(X, K_{k,n-k})$-exchangeable from $\sigma$ by considering swaps along edges in $E(C)$. Suppose now that it holds for all $|L|\le l-1$, where $l\ge 3$.   Because $X$ is $(k-1)$-connected and $|M|<k-1$, the graph  $X'=X|_{V(X)\setminus M}$ is connected.
 
 We now show that $V(C)$ is a proper subset of $V(X')$. Suppose to the contrary that $V(C)=V(X')$, then it also holds that $E(C)=E(X')$ since $C$ is a shortest odd cycle in $X$. Thus, $M\neq \varnothing$ since $X$ is not a cycle. Because $X'$ is $(k-1-|M|)$-connected, we have $k-1-|M|\le 2$, i.e.,  $|M|\ge k-3$. So, we have $|M|=k-3$ since $|L|=l\ge 3$. Because $X$ is $(k-1)$-connected, the minimum degree of $X$ is at least $k-1$, so every vertices in $V(C)$ is adjacent to all vertices in $M$, which implies that $X$ contains a triangle. The contradiction arises since $|V(C)|= n-|M| \ge n-(k-2) \ge k+2 >3$.

 Let  $c_0$ be a vertex in $N(V(C))\cap V(X')$. By the structure of $\FS(C_t,K_{k,t-k})$ described in Lemma \ref{23}, there is a sequence  $S$  of swaps, involves only edges in $E(C)$, that transforms $\sigma$ into a new bijection $\tau$ such that $\tau^{-1}(1)\tau^{-1}(2), c_0 \tau^{-1}(a)\in E(X)$ for some $a$ satisfying $\sigma^{-1}(a) \in L\setminus \{\sigma^{-1}(1),\sigma^{-1}(2)\}$.  Let $\tau'=\sigma\circ(c_0\ \tau^{-1}(a))$ denote the bijection obtained from $\tau$ by swapping the edge $c_0 \tau^{-1}(a)$.  Then we have  $|L(\tau')|=l-1$. By induction hypothesis, there exists a sequence  of swaps that transforms $\tau'$ into $\tau''=(1 \ 2) \circ \tau'$. Then the sequence  $S^{-1}$ transforms $\tau=\tau'' \circ  (c_0\ \tau^{-1}(a))$ into the desired $(1\ 2)\circ \sigma$.

\vskip 2mm
\noindent\textbf{Case 2.}  Exactly one of $\sigma^{-1}(1)$ and $\sigma^{-1}(2)$ is in $V(C)$.
\vskip 2mm

 Assume that $\sigma^{-1}(1)\in V(C)$. We consider the following two subcases separately. 

\vskip 2mm
\noindent\textbf{Subcase 2.1.} The vertex set $ V(C)\setminus \sigma^{-1}(\{1,\dots,k\})$ is non-empty.
\vskip 2mm

 Let $c_1$ be a vertex in $ V(C)\setminus \sigma^{-1}(\{1,\dots,k\})$. By the structure of $\FS(C_t,K_{k,t-k})$ described in Lemma \ref{23},  there is a sequence  $S$  of swaps, involves only the edges in $E(C)$, that transforms $\sigma$ into a new bijection $\tau$ satisfying $\tau^{-1}(2)c_1 \in E(X)$, $\tau^{-1}(1)c_1\in E(C)$. So, the bijection $\tau'=\tau\circ (\sigma^{-1}(2) \ c_1)$ satisfies that $\tau'^{-1}(1)\tau'^{-1}(2) \in E(C)$. By Case 1, there is a sequence of swaps that transforms $\tau'$ into $\tau''=(1 \ 2)\circ \tau'$. Then the sequence $S^{-1}$ transforms $\tau'' \circ (c_1 \ \sigma^{-1}(1))$ into the desired $(1 \ 2)\circ \sigma$.
 
\vskip 2mm
\noindent\textbf{Subcase 2.2.}  The vertex set $ V(C)\setminus \sigma^{-1}(\{1,\dots,k\})$ is empty.
\vskip 2mm

We first observe that $|L|=|V(C)|\ge 3$, which implies that $|M\cup \{\sigma^{-1}(1)\}|=k-|L|+1<k-1$. Because $X$ is $(k-1)$-connected, the graph $X''=X|_{V(X)\setminus (M\cup \{\sigma^{-1}(1)\})}$ is connected.

 Because $n-k\ge k$, we have the relation 
$$|V(X'')|\ge n-k+2 >   k-1  \ge |V(C)\setminus \{\sigma^{-1}(1)\}|,$$
which implies that $V(C)\setminus \{\sigma^{-1}(1)\}$ is a proper subset of $V(X'')$. So, there exists a vertex $c_2\in N(V(C)\setminus \{\sigma^{-1}(1)\})\cap V(X'')$ and let $c_3\in V(C)\setminus \{\sigma^{-1}(1)\}$ be a vertex such that $c_2 c_3 \in E(X'')$. Then the bijection $\tau=\sigma\circ (c_2 \ c_3)$ satisfies  $V(C)\setminus \tau^{-1}(\{1,\dots,k\})\neq \varnothing$, which implies that $\tau$ satisfies the condition in Subcase 2.1. Thus, there is a sequence of swaps that transforms $\tau$ into $\tau'=(1 \ 2)\circ \tau$. Then the bijection $(c_2 \ c_3)\circ \tau'$ equals to the desired $(1 \ 2)\circ \sigma$.

\vskip 2mm
\noindent\textbf{Case 3.} None of $\sigma^{-1}(1)$ and $\sigma^{-1}(2)$ is in $V(C)$.
\vskip 2mm

We consider the following two subcases separately.

\vskip 2mm
\noindent\textbf{Subcase 3.1.} The vertex set $ V(C)\setminus \sigma^{-1}(\{1,\dots,k\})$ is non-empty.
\vskip 2mm

Because $X$ is $(k-1)$-connected, the graph  $X'''=X|_{V(X)\setminus \sigma^{-1}(\{3,\dots,k\})}$ is connected. 
Let $x_1 x_2\cdots x_p$ be a shortest path  in $X'''$ between the two sets $\{\sigma^{-1}(1),\sigma^{-1}(2)\}$ and  $V(C)\setminus \sigma^{-1}(\{1,\dots,k\})$. Assume that $x_1=\sigma^{-1}(1)$. Let $x_0=\sigma^{-1}(2)$ and denote by $S$ the sequence of swaps along the edges $x_1x_2,\dots,x_{t-1}x_t, x_0x_1,\dots,x_{t-2}x_{t-1}$. The sequence $S$ transforms $\sigma$ into a new bijection $\tau$ satisfying $\tau^{-1}(1)\tau^{-1}(2)\in E(X)$ and only one of $\tau^{-1}(1)$ and $\tau^{-1}(2)$ lies in $V(C)$. By Case 2, there is a sequence of swaps that transforms $\tau$ into $\tau'=(1 \ 2)\circ \tau$. Then the sequence $S^{-1}$ transforms $\tau'$ into the desired $(1\ 2)\circ \sigma$.

\vskip 2mm
\noindent\textbf{Subcase 3.2.} The vertex set $V(C)\setminus \sigma^{-1}(\{1,\dots,k\})$ is empty.
\vskip 2mm

We first observe that $|L|=|V(C)|\ge 3$, which implies that $|M|=k-|L|<k-1$. Because $X$ is $(k-1)$-connected, the graph $X'=X|_{V(X)\setminus M}$ is connected.

 Because $n-k\ge k$, we have 
$$|V(X')|\ge n-k+2 >   k-1  \ge |V(C)|,$$
which implies that $V(C)$ is a proper subset of $V(X')$. So, there exists a vertex $c_2\in N(V(C))\cap V(X')$ and let $c_3\in V(C)$ be a vertex such that $c_2 c_3 \in E(X')$. Then the new bijection $\tau=\sigma\circ (c_2 \ c_3)$ satisfies  $V(C)\setminus \tau^{-1}(\{1,\dots,k\})\neq \varnothing$, which implies that $\tau$ satisfies the condition in Subcase 3.1. Thus, there is a sequence of swaps that transforms $\tau$ into $\tau'=(1 \ 2)\circ \tau$. It is easy to verify that the bijection $(c_2 \ c_3)\circ \tau'$ equals to the desired $(1 \ 2)\circ \sigma$.
 \end{proof}

\begin{proof}[\bfseries{Proof of Theorem \ref{TTT}}]
By Lemmas  \ref{lem:exchangeable} and \ref{LL}, the graph $\FS(X,K_{k,n-k})$ is connected if and only if $\FS(X,L_{k,n-k})$ is connected, where $L_{k,n-k}$ is the graph obtained by adding the edges $12, 13, \dots, 1k$ to $K_{k,n-k}$. It is easy to observe that $L_{k,n-k}$ contains the graph $S^+_n$ as a spanning subgraph.

 By Lemma \ref{31}, the condition, that $X$ is a $(k-1)$-connected non-bipartite graph and  not a cycle, guarantees the connectedness of $\FS(X,S^+_n)$ for $X\neq \varTheta$. By Lemma \ref{asd}, $\FS(X,L_{k,n-k})$ is connected and thus $\FS(X,K_{k,n-k})$ is connected for $X\neq \varTheta$. If $X=\varTheta$, then $k$ equals to $3$ since $n=7$ and $n-k \ge k$. It can be verified by computer that the graph $\FS(\varTheta,K_{3,4})$ is connected. 
\end{proof}

\begin{proof}[\bfseries{Proof of Theorem \ref{TTTT}}]

 For the disconnectedness part, it is well known that if 
$$p\le \frac{\log n-c(n)}{n}$$
 for some $c(n)\to +\infty$ arbitrary slowly, then the random graph $X\in \mathcal{G}(n,p)$ is disconnected with high probability. Thus, the graph $\FS(X,K_{k,n-k})$ is disconnected with high probability by Lemma \ref{20}. 

 For the connectedness part, it is well known that if 
$$p\ge \frac{\log n+k\log \log n+c(n)}{n}$$
 for some $c(n)\to +\infty$ arbitrary slowly, then the random graph $X\in \mathcal{G}(n,p)$ is a $(k+1)$-connected, non-bipartite graph, and is not a cycle with high probability. Thus, by Theorems \ref{WW}, \ref{TT} and \ref{TTT}, we can see that the graph $\FS(X,K_{k,n-k})$ is connected with high probability.  
\end{proof}

\section{Open problems}\label{s5}

Based on Theorems \ref{T}, \ref{TT} and \ref{TTT}, we have the following conjecture.
\begin{conj}\label{C}
Let $Y$ be a graph on $n\ge 2k\ge 4$ vertices. The graph $\FS(K_{k,n-k},Y)$ is connected if and only if $Y$ is a connected non-bipartite graph without non-trivial $k$-bridge and $Y\not=C_n$.
\end{conj}

 The disconnectedness part of Conjecture \ref{C} is guaranteed by Theorem \ref{T}.
 And Theorems \ref{TT} and \ref{TTT} are  weaker forms of  the connectedness part of Conjecture \ref{C}.\\

 Wilson \cite {W} gave a sufficient and necessary condition for $\mathsf{FS}(S_n,Y)$ to be connected, see  Theorem \ref{WW}. In addition, he also proved that  $\mathsf{FS}(S_n,Y)$  has exactly $2$ components for any   $2$-connected bipartite graph $Y$. Based on his result and Theorem \ref{TT}, we have the following conjecture.
\begin{conj}\label{cc2}

Let  $Y$ be a connected bipartite graph on $n\ge 5$ vertices with no non-trivial cut edge and $Y\not=C_n$, then the graph $\FS(K_{2,n-2},Y)$ has exactly $2$ components.
\end{conj}
  
 Let $Y$ be a bipartite graph on $n$ vertices with bipartition $S$ and $T$ of size $s$ and $t$, respectively.
 Alon, Defant and Kravitz \cite{ADK} proved that the graph $\FS(K_{2,n-2},K_{s,t})$ has exactly $2$ components. By Lemma \ref{lem:exchangeable}, in order to prove Conjecture \ref{cc2}, it suffices to show that any two vertices $u\in S$ and $v\in T$ in $Y$ are $(K_{2,n-2},Y)$-exchangeable from any bijection $\sigma: V(K_{2,n-2}) \mapsto V(Y)$ satisfying $\sigma^{-1}(u) \sigma^{-1}(v) \in E(K_{2,n-2})$.

\section*{Date availability}

No data was used for the research described in the article.

\section*{Acknowledgments}

 This research was supported by NSFC under grant numbers  12161141003 and 11931006.  

%\newpage

\end{document}